\documentclass{article}%
\usepackage{amsmath,amsfonts,amssymb,graphicx}%

\newtheorem{thm}{Theorem}

\newtheorem{cor}[thm]{Corollary}
\newtheorem{pro}[thm]{Proposition}

\def\P{\mathcal{P}}

\def\Bbb{\mathbb}

\def\proof{{\noindent \bf Proof. \hspace{.01in}}}
\newcommand{\qed}{\hspace{.1in} \vrule height 7pt width 5pt depth 0pt \medskip}

\begin{document}


\title{\Large Comparison of the upper bounds for the extreme points of
the polytopes of line-stochastic tensors}
\author{
Fuzhen Zhang${}^{\,\rm a}$\thanks{zhang@nova.edu}
,\;
Xiao-Dong Zhang${}^{\,\rm b}$\thanks{Corresponding author. xiaodong@sjtu.edu.cn}
\\
\footnotesize{${}^{\,\rm a}$Department of Mathematics, Nova Southeastern University, Fort Lauderdale, USA}\\
\footnotesize{${}^{\,\rm b}$School of mathematical sciences, Shanghai Jiao Tong University, Shanghai, China}}

\date{}
\maketitle
In memory of Professor Ky Fan.
\bigskip
 \hrule
\bigskip

\noindent {\bf Abstract}
We call a real multi-dimensional array a {\em tensor} for short.
 In enumerating vertices of the
polytopes of stochastic tensors, different approaches
have been used: {(1)}  Combinatorial method via Latin squares;
{(2)} Analytic (topological)  approach by using hyperplanes;  {(3)} Computational geometry (polytope theory) approach; and
(4) Optimization (linear programming) approach. As all these approaches are
 worthy of consideration and investigation in the enumeration problem, various bounds have been obtained.
This note is to compare the existing upper bounds arose from different approaches.

\medskip


\bigskip
\hrule

\date{}
\maketitle

\bigskip

Polytopes play an important role in mathematics and applications, most notably in discrete geometry and linear programming.
Here by a {\em polytope} we mean a bounded convex set
contained in a Euclidean space $\Bbb R^d$  that is generated by  finitely many points. In other words,
if  $\P \subset \Bbb R^d$ is a {polytope}, then $\P$ is the convex hull of a finite set of points
in $\Bbb R^d$ (see, e.g., \cite[p.\,8]{Bar02} or \cite[p.\,4]{Zie95}).
The Krein-Milman Theorem (see, e.g., \cite[p.\,121]{Bar02}) ensures that every polytope
is the convex hull of its vertices (extreme points). It is a fundamental and central  question in the polytope theory to determine the number and structures of the vertices (or faces of higher dimensions) for a given polytope, and this is
  an extremely difficult problem in general
(see, e.g., \cite{Bro83} or \cite[p.\,254]{Zie95}).

 A doubly stochastic matrix is a nonnegative square matrix in which each row sum as well as each column sum is equal to one.
The  classical  Birkhoff polytope $\mathcal{B}_n$ consists
of $n\times n$ doubly stochastic matrices, and the celebrated  Birkhoff-von Neumann Theorem states that $\mathcal{B}_n$ is the convex hull of all $n\times n$  permutation matrices (see, e.g., \cite[p.\,159]{ZFZbook11}). In other words,
as a polytope in $\Bbb R^{n^2}\hspace{-.05in},$    $\;\mathcal{B}_n$ has  $n!$ vertices (see, e.g., \cite[p.\,20]{Zie95}).  Carath\'{e}odory's theorem ensures that
every  $n\times n$ doubly stochastic matrix can be written as a convex combination of at most $n^2-2n+2$ permutation matrices. 
Geometrical and combinatorial properties of  the Birkhoff polytope  have been extensively studied; see, e.g.,  \cite{Ahmed08, BrCs75, BrCs75laa,  BrCs91, ChenQi19, CLN14,  Paf15, LL14};
see also \cite[pp.\,47--52]{MOA11} for a brief account.

We are concerned with the polytopes of stochastic tensors.
We simply call  a real multidimensional array (i.e., matrix of higher order or hypermatrix) a {\em tensor}.
See, e.g.,
\cite{WeiBook16, Qi2018Book, QLbook17},  for the theory and applications of tensors.
By a {\em stochastic tensor} we mean a tensor of certain stochastic properties (such as
line-stochasticity, defined below,  and plane-stochasticity, etc.).

Let ${n_1, n_2, \dots, n_d}$ be positive integers. As usual, we write
$$A=(a_{i_1i_2\dots i_d}), \quad  {\rm where}\;
a_{i_1i_2\dots i_d}\in \Bbb R,\;
i_k=1, 2,  \dots, n_k, \; k=1, 2,  \dots, d,$$  for an $n_1\times n_2\times  \cdots \times n_d$  tensor   $A$ of order $d$ (the number of indices).
A tensor  $A=(a_{i_1i_2\dots i_d})$ may be regarded as an element in
$\Bbb R^{n_1n_2 \cdots n_d}$.
The tensors of order 1 (i.e., $d=1$) are the vectors in $\Bbb R^{n_1}$, while the 2nd order tensors are the usual  matrices. A regular Rubik's Cube may be regarded as a $3\times 3\times 3$
 {tensor}.
 An $m\times n\times 3$ tensor has three (frontal) layers of $m\times n$ matrices, and 
 it may be identified
 with an $m\times (3n)$ rectangular matrix.
If $n_1= n_2=\dots= n_d=n$, we say that
$A$ has  order $d$ and dimension $n$  or  $A$ is an $\overbrace{n\times   \cdots \times n}^d$ tensor.
(Note: the terms order and dimension may be defined differently in other texts.)

 For a nonnegative tensor $A=(a_{i_1i_2\dots i_d})$
of order $d$ and dimension $n$, we say that $A$ is {\em line-stochastic} \cite{FiSw85}
  if the sum of  the entries on each line is   1, that is,
$$\sum_{i=1}^n a_{\cdots i\cdots }=1.$$

An $n\times n$ doubly stochastic matrix,  in particular, a permutation matrix,
is a line-stochastic tensor of  order 2 and dimension $n$.
The   Birkhoff polytope $\mathcal{B}_n$
of $n\times n$ doubly stochastic matrices is generated (via convex combinations) by
exactly the permutation matrices. However, for higher order tensors, the situation is
very different and it can be   much more complicated.  Let

$$ Q=\frac12 \left [\begin{array}{ccc}
0 & 1 & 1\\
1 & 1 & 0\\
1 & 0 & 1
\end{array} \right |  \begin{array}{ccc}
1 & 1 & 0 \\
0 &1 & 1\\
1 & 0 & 1
\end{array}  \left |\begin{array}{ccc}
1 & 0 & 1 \\
1 & 0 & 1\\
0& 2 & 0
\end{array} \right ].$$
Then $Q$ is  a line-stochastic tensor of  order 3 and dimension $3$. One may verify that
$Q$ is not a convex combination of (0,1)-tensors  (i.e., permutation tensors).

We
 denote by $\mathcal{L}_n$ the polytope of  the $n\times n\times n$ (triply) line-stochastic tensors.
(One may define stochastic tensors of higher orders, such as triply plane-stochastic tensors.)
The Krein-Milman Theorem asserts that  every polytope
is the convex hull of its extreme points. Then what are those  for
$\mathcal{L}_n$? The above example $Q$  shows  that $\mathcal{L}_3$ has
extreme points other than the (0,1)-tensors.

It is an interesting problem and uneasy task (see, e.g., \cite{LZZ17}) to determine the extreme points
  for a polytope of stochastic tensors; see \cite{KLX16} for more information
  on this topic. For other aspects such as the permanents of tensors,
 see \cite{WZ17Per} and the references therein.
 For general tensors and their properties, the reader is referred to the books \cite{ WeiBook16, Qi2018Book, QLbook17}.

   In enumerating vertices of the
polytope $\mathcal{L}_n$, different approaches
have been undertaken: {(1)}.  Combinatorial method via Latin squares (see, e.g.,
\cite[Theorem~0.1]{Mayathesis03} or \cite[Theorem 2.0.10]{Mayathesis04});
{(2)}. Analytic  approach by using hyperplane and induction \cite{ChangPZ16};  {(3)}. Computational geometry approach \cite{LZZ17}; and
(4). Optimization (operation research) approach \cite{ZZOptim20}. As all these approaches to the enumeration are
 worthy of investigation, various bounds have been obtained.
We compare the existing bounds arose from different approaches.


Let $f_0(\mathcal{L}_n)$ be the number of vertices of $\mathcal{L}_n$.
(Note:  $f_i(\mathcal{P})$ usually denotes the number of faces of dimension $i$ of polytope
$\mathcal{P}$.)  We have seen
the estimation of $f_0(\mathcal{L}_n)$
 in various ways.
By a combinatorial method using Latin squares,
Ahmed,  De Loera, and  Hemmecke (see \cite[Theorem 2.0.10]{Mayathesis04} or \cite[Theorem~0.1]{Mayathesis03})
 gave an explicit lower bound $\frac{(n!)^{2n}}{n^{n^2}}$. A sharper lower bound is immediate
  by noticing that the number of Latin squares of order $n$, denoted by $L(n)$, is equal to the number of
 $n\times n\times n$ line-stochastic (0,1)-tensors (see \cite{JurRys68} or \cite[pp.\,159--161]{Lint92}). Observe  that every (0-1)-stochastic tensor is
 an extreme point. So
 $$\frac{(n!)^{2n}}{n^{n^2}}\leq L(n)\leq f_0(\mathcal{L}_n).$$
This brilliant idea is seen in Jurkat and Ryser \cite{JurRys68}.
For the case of $n=3$, one may identity (via one-to-one mapping)  a $3\times 3$ Latin square $S$ with a $3\times 3\times 3$ tensor cube $T$:
 If $(i, j)$-entry of the Latin square is $k$,
  $1\leq i, j, k \leq 3$, then let $t_{ijk}=1$ and all other $t_{pqr}=0$.  For example,
$$S=\left [\begin{array}{ccc}
1 & 2 & 3 \\
2 & 3 & 1\\
3 & 1 & 2
\end{array} \right ] \mapsto T=\left [\begin{array}{ccc}
1 & 0 & 0\\
0 & 0 & 1\\
0 & 1 & 0
\end{array} \right |  \begin{array}{ccc}
0 & 1 & 0 \\
1 & 0 & 0\\
0 & 0 & 1
\end{array}  \left |\begin{array}{ccc}
0 & 0 & 1 \\
0 & 1 & 0\\
1 & 0 & 0
\end{array} \right ].$$

Other approaches and upper bounds are recapped below.

By an
analytic and topological approach using hyperplanes and induction,
 Chang, Paksoy, and Zhang \cite{ChangPZ16} obtained an upper bound.

\begin{thm}
 [Theorem~4.1 \cite{ChangPZ16}]\label{CPZ16}
 Let $f_0(\mathcal{L}_n)$ be the number of vertices of the polytope $\mathcal{L}_n$ of
 the $n\times n\times n$ line-stochastic tensors.
Then
\begin{equation}\label{CPZ16}
f_0(\mathcal{L}_n)\leq \frac{1}{n^3}  {p(n)\choose n^3-1},
\end{equation}
\mbox{where $p(n)=n^3+6n^2-6n+2$}.
\end{thm}

By a computational geometry approach using the  McMullen Upper Bound Theorem (UBT) \cite{McM70} (see also, e.g., \cite[p.\,90]{Bro83})  for  polytopes.
Li, Zhang and Zhang
\cite{LZZ17}
showed an upper bound.

\begin{thm}[Theorem~2 \cite{LZZ17}]\label{LZZ17}
Let $f_0(\mathcal{L}_n)$ be the number of vertices of the polytope $\mathcal{L}_n$ of
 the $n\times n\times n$ line-stochastic tensors. Then
\begin{equation}\label{LZZ17}
 f_0(\mathcal{L}_n)\leq \left (\hspace{-.08in} \begin{array}{c}
n^3- \lfloor \frac{(n-1)^3+1}{2}\rfloor\\
 3n^2-3n+1
 \end{array} \hspace{-.08in} \right )+\left (\hspace{-.08in}  \begin{array}{c}
n^3- \lfloor \frac{(n-1)^3+2}{2}\rfloor\\
 3n^2-3n+1
 \end{array} \hspace{-.08in} \right  ).
\end{equation}
\end{thm}

It is shown in \cite[Proposition~3]{LZZ17}
that  the upper bound in (\ref{LZZ17})  is better (shaper) than the  one in (\ref{CPZ16}). However, the lower bound derived  by computational geometry approach (lower bound theorem)
 is no better in general.

By an approach of
 optimization and linear programming, Zhang and Zhang presented another upper bound for
 $f_0(\mathcal{L}_n)$.

\begin{thm}[Theorem~3.4 \cite{ZZOptim20}]\label{ZZ18OR}
Let $f_0(\mathcal{L}_n)$ be the number of vertices of the polytope $\mathcal{L}_n$ of the
 $n\times n\times n$ line-stochastic tensors.  Then
 \begin{equation}\label{ZZ20a}
 f_0(\mathcal{L}_n)\leq \sum_{k=n^2}^{3n^2-3n+1} {{n^3}\choose {k}}.
 \end{equation}
\end{thm}

It was asked \cite{LZZ17}
 and has remained unanswered
whether there is a comparison between the bounds in (\ref{LZZ17}) and  (\ref{ZZ20a}).
That is, would the upper bound by the computational geometry  be better (or worse)
than the one by
 optimization and linear programming? We answer the question now.

\begin{thm}
  The   upper bound in (\ref{LZZ17}) is sharper  than that in  (\ref{ZZ20a}). In fact, for $n\ge 2$,
{\small $$
\left (\hspace{-.08in} \begin{array}{c}
n^3- \lfloor \frac{(n-1)^3+1}{2}\rfloor\\
 3n^2-3n+1
 \end{array} \hspace{-.08in} \right )+\left (\hspace{-.08in}  \begin{array}{c}
n^3- \lfloor \frac{(n-1)^3+2}{2}\rfloor\\
 3n^2-3n+1
 \end{array} \hspace{-.08in} \right  )<{{n^3}\choose {3n^2-3n+1}}
 <
 \sum_{k=n^2}^{3n^2-3n+1} {{n^3}\choose {k}}.
$$}
\end{thm}

\proof The second inequality is obvious. We show the first one.
We first prove  that
if $a, b, k$ are positive integers with $k\ge 2$, $a>b$, and $b(k+1)> a+k$, then
\begin{equation}\label{Ineq:2ab}
2{a\choose b} <{{a+k}\choose b}.
\end{equation}
This is justified as follows.
 Since $k\ge 2$, $a>b$, and $b(k+1)> a+k$, we have
 $$\left(\frac{a+k}{a+k-b}\right)^k=\left(1+\frac{b}{a+k-b}\right)^k>1+\frac{bk}{a+k-b}>2.$$
On the other hand,
noticing that  $$
\frac{a+k}{a+k-b}<\frac{a+k-1}{a+k-b-1}<\cdots<\frac{a+1}{a-b+1},
$$
we get
$$\frac{(a+k)(a+k-1)\cdots(a+1)}{(a+k-b)(a+k-b-1)\cdots(a-b+1)}>\left(\frac{a+k}{a+k-b}\right)^k>2.$$
Hence
\begin{eqnarray*}
\frac{{{a+k}\choose{b}}}
{{a\choose b}} & = &
\frac{\frac{(a+k)!}{b!(a+k-b)!}}{\frac{a!}{b!(a-b)!}}\\
& = & \frac{(a+k)(a+k-1)\cdots(a+1)}{(a+k-b)(a+k-b-1)\cdots(a-b+1)}>2.
\end{eqnarray*}

Now for the claimed inequality in the theorem, it is easy to check the case of $n=2$ by direct computations: $\left ( {7 \atop 7} \right ) +\left ( {7 \atop 7} \right )< \left ( {8 \atop 7} \right )<\sum_{k=4}^7
\left ( {8 \atop k} \right ).$

Let $n\geq 3$.
Then  $n^3-\lfloor\frac{(n-1)^3+1}{2}\rfloor< n^3-n$. With
  ${a\choose b} <{{a+1}\choose b}$, we derive
\begin{eqnarray*}
\lefteqn{\left(\begin{array}{c} n^3-\lfloor\frac{(n-1)^3+1}{2}\rfloor\\
3n^2-3n+1\end{array}\right)+\left(\begin{array}{c} n^3-\lfloor\frac{(n-1)^3+2}{2}\rfloor\\
3n^2-3n+1\end{array}\right)}\\
& & \qquad \le  2 \left(\begin{array}{c} n^3-\lfloor\frac{(n-1)^3+1}{2}\rfloor\\
3n^2-3n+1\end{array}\right)\\
& & \qquad  <2\left(\begin{array}{c} n^3-n\\
 3n^2-3n+1\end{array}\right).
\end{eqnarray*}

In (\ref{Ineq:2ab}), we set $a= n^3-n$, $b=3n^2-3n+1$, and $k=n$. Then $a>b$, $k\ge 2$, and $b(k+1)=(3n^2-3n+1)(n+1)>n^3-n+n=a+k$. It follows that
$$2\left(\begin{array}{c} n^3-n\\
3n^2-3n+1\end{array}\right)\\
  <\left(\begin{array}{c} n^3\\
3n^2-3n+1\end{array}\right). \quad   \qed $$

By
  characterization of extreme points of polytopes described through linear inequalities
  (half-spaces),
 Zhang and Zhang gave an estimate of $f_0(\mathcal{L}_n)$  in
  \cite{ZZOptim20}.

\begin{thm}[Theorem 3.6  \cite{ZZOptim20}]\label{ZZ18}
Let $f_0(\mathcal{L}_n)$ be the number of vertices of the polytope $\mathcal{L}_n$ of the
 $n\times n\times n$ line-stochastic tensors. Then
\begin{equation}\label{ZZ20b}
f_0(\mathcal{L}_n)\leq {{n^3+3n^2-3n+1}\choose {n^3}}.
\end{equation}
\end{thm}

However, this upper bound in (\ref{ZZ20b}) is no better than that in (\ref{ZZ20a}).

\begin{thm}\label{ZZ18}
 The upper bound in (\ref{ZZ20a})    is sharper than that in (\ref{ZZ20b}). That is,
$$
\sum_{k=n^2}^{3n^2-3n+1} {{n^3}\choose {k}}<{{n^3+3n^2-3n+1}\choose {n^3}}.
$$
\end{thm}

\begin{proof}
If $n=1$, it is obvious. If $n=2$, then
 $\left ( {8 \atop 4} \right ) +\left ( {8 \atop 5} \right )+\left ( {8 \atop 6} \right )+\left ( {8 \atop 7} \right )
 <
\frac{1}{2^3}\left ( {15 \atop 8} \right ).$

Let $n\geq 3$. With the identity ${a\choose b}+{a\choose {b+1}}={{a+1}\choose {b+1}}$, we can show that
$$\sum_{k=0}^m\left(\begin{array}{c}
a\\
b+k\end{array}\right)\le\left(\begin{array}{c}
a+m\\
b+m\end{array}\right).
$$
Now we compute
\begin{eqnarray*}
\sum_{k=n^2}^{3n^2-3n+1}\left(\begin{array}{c}
n^3\\
k\end{array}\right)
 & = & \sum_{k=0}^{2n^2-3n+1}\left(\begin{array}{c}
n^3\\
n^2+k\end{array}\right)\\
& \le & \left(\begin{array}{c}
n^3+2n^2-3n+1\\
n^2+2n^2-3n+1\end{array}\right)\\
& = & \left(\begin{array}{c}
n^3+2n^2-3n+1\\
3n^2-3n+1\end{array}\right)\\
& < & \left(\begin{array}{c}
n^3+3n^2-3n+1\\
3n^2-3n+1\end{array}\right)\\
& = & \left(\begin{array}{c}
n^3+3n^2-3n+1\\
n^3\end{array}\right). \quad \qed
\end{eqnarray*}
\end{proof}

\begin{pro}\label{inequality}
The upper bound in (\ref{ZZ20b}) is sharper  than that in (\ref{CPZ16}). That is,
for $n\ge 2$,
\begin{equation*}\label{ineqality-1}
{{n^3+3n^2-3n+1}\choose {n^3}}< \frac{1}{n^3} {{n^3+6n^2-6n+2}\choose {n^3-1}}.
\end{equation*}
\end{pro}

\begin{proof}
 If $n=2$, then
 $${{n^3+3n^2-3n+1}\choose {n^3}}=\binom{15}{8}<\frac{1}{2^3}\binom{22}{7}= \frac{1}{n^3} {{n^3+6n^2-6n+2}\choose {n^3-1}}.$$
 Let $n\ge 3$.  Note that for positive integers $a$, $ b$, and $x$,  we have
  \begin{equation}\label{Eq:7}
  \binom{a}{b}<\binom{a+1}{b+1}\; \mbox{if $a>b$},\quad
 \binom{a}{b}\le \binom{a}{b+1}\; \mbox{if  $a\ge 2b+1$,}
 \end{equation}
and
$$ \frac{a+1}{b+1}>\frac{a+2}{b+2}>\cdots> \frac{a+x}{b+x} \;\; \mbox{if $a>b$ and  $x>2$}.$$
We obtain $$\frac{a+x}{b+x}\cdot\frac{a+x-1}{b+x-1}\cdots\frac{a+1}{b+1}>\left(\frac{a+x}{b+x}\right)^x=\left(1+\frac{a-b}{b+x}\right)^x.$$
  Since $n\geq 3$, we have
  $n(n^3-1)\ge 2(6n^2-6n+3)$.  Setting
   $a=n^3+5n^2-6n+2$, $b=5n^2-6n+3$,
 and $x=n^2$ in the above discussion, we derive
$$\left(1+\frac{a-b}{b+x}\right)^x=\left[\left(1+\frac{n^3-1}{6n^2-6n+3}\right)^{\!n}\,\right]^{n}
>\left(1+\frac{n(n^3-1)}{6n^2-6n+3}\right)^{\!n}\ge 3^n\ge n^3.$$
Hence
  $$\binom{a+x}{b+x}=\frac{a+x}{b+x}\cdot\frac{a+x-1}{b+x-1}\cdots\frac{a+1}{b+1}\binom{a}{b}>\left(\frac{a+x}{b+x}\right)^x\binom{a}{b}>n^3\binom{a}{b}.$$
 Since $n\geq 3$,
  $n^3+3n^2-3n+1> 2(3n^2-3n+1)+1$.  Using (\ref{Eq:7}),  we have
  \begin{eqnarray*}
  {{n^3+3n^2-3n+1}\choose {n^3}} &=& {{n^3+3n^2-3n+1}\choose {3n^2-3n+1}}\\
  &<& {{n^3+3n^2-3n+1}\choose {3n^2-3n+2}}  \\
  &<& {{n^3+3n^2-3n+1 +(2n^2-3n+1)}\choose {3n^2-3n+2+(2n^2-3n+1)}}\\
  &=& {{n^3+5n^2-6n+2}\choose {5n^2-6n+3}}= {a\choose b}\\
  &<& \frac{1}{n^3}{{n^3+5n^2-6n+2 + (n^2)}\choose {5n^2-6n+3+(n^2)}}\\
  &= & \frac{1}{n^3} {{n^3+6n^2-6n+2 }\choose {6n^2-6n+3}} \\
  &=& \frac{1}{n^3}{{n^3+6n^2-6n+2}\choose {n^3-1}}. \;\; \qed
  \end{eqnarray*}
\end{proof}

We summarize the comparisons of the upper bounds for $f_0(\mathcal{L}_n)$  as follows.
\begin{cor}
Let $f_0(\mathcal{L}_n)$ be the number of vertices of the polytope $\mathcal{L}_n$ of the
 $n\times n\times n$ line-stochastic tensors with $n\ge 2$. Then
\begin{eqnarray*}
\lefteqn{f_0(\mathcal{L}_n)} \\
& \leq &
\left (\hspace{-.08in} \begin{array}{c}
n^3- \lfloor \frac{(n-1)^3+1}{2}\rfloor\\
 3n^2-3n+1
 \end{array} \hspace{-.08in} \right )+\left (\hspace{-.08in}  \begin{array}{c}
n^3- \lfloor \frac{(n-1)^3+2}{2}\rfloor\\
 3n^2-3n+1
 \end{array} \hspace{-.08in} \right  ) \; \mbox{\rm (by polytope theory)}\\
 & < & \sum_{k=n^2}^{3n^2-3n+1} {{n^3}\choose {k}} \;\; \mbox{\rm (by optimization and linear programing)}\\
 &  < &
 {{n^3+3n^2-3n+1}\choose {n^3}} \;\; \mbox{\rm (by half-spaces)}\\
& < & \frac{1}{n^3}  {n^3+6n^2-6n+2 \choose n^3-1}\;
  \; (\mbox{\rm by topology and hyperplanes}). \\
 \end{eqnarray*}
\end{cor}

 As the authors previously pointed out that these
 upper   bounds are very large when $n$ is large and the bounds are loose due to the structures of the polytopes.

\bigskip
\noindent
{\bf Acknowledgement.} The authors thank the referee for carefully reading the manuscript.
   Xiao-Dong Zhang's work
  was partially supported by National Natural Science Foundations of China
   No. 11971311 and No. 12026230.

\end{document}